\documentclass[10pt]{amsart} 

\usepackage[utf8]{inputenc} 
\usepackage{geometry} 
\usepackage{graphicx} 
\usepackage{booktabs} 
\usepackage{array} 
\usepackage{paralist} 
\usepackage{verbatim} 
\usepackage{tikz}
\usetikzlibrary{arrows}
\usetikzlibrary{decorations.markings}
\usepackage{subfig}
\usepackage{tabularx}
\usepackage{enumitem}
\usepackage{amsmath,amsthm}
\usepackage{amssymb,amsfonts}
\usepackage{fancyhdr} 
\usepackage{pgf,tikz}
\usepackage{caption}
\usepackage{tikz-cd}
\usepackage{pst-node}
\usepackage{tikzscale}
\usetikzlibrary{patterns}
\usepackage{rotating}
\usepackage{float}

\usepackage{color}   
\usepackage{hyperref}
\hypersetup{
    colorlinks=true, 
    linktoc=all,     
    linkcolor=blue,  
    citecolor=blue,
    filecolor=blue,
    urlcolor=blue
}

\usepackage[T1]{fontenc}
\usepackage{csquotes}
\usepackage[style=numeric,defernumbers,backend=bibtex]{biblatex}
\usepackage{hyperref}
\usepackage{nameref}
\defbibfilter{other}{
  not type=article
  and
  not type=book
  and
  not type=collection
}

\usepackage{amssymb,amsmath,latexsym,amsthm}
\usepackage{faktor} 		

\theoremstyle{plain}                    
\newtheorem{thm}{Theorem}[section]
\newtheorem{lem}[thm]{Lemma}
\newtheorem{prop}[thm]{Proposition}
\newtheorem{cor}[thm]{Corollary}

\theoremstyle{definition}
\newtheorem{defn}[thm]{Definition}

\theoremstyle{remark}
\newtheorem{rmk}[thm]{Remark}

\numberwithin{equation}{section}

\newcommand{\R}{\mathbb{R}}

\newcommand{\C}{\mathbb{C}}

\newcommand{\hyp}{\mathbb{H}}

\newcommand{\cp}{\mathbb{C}\mathbb{P}^1}

\newcommand{\pslc}{\mathrm{PSL}_2\C}

\newcommand{\pslr}{\mathrm{PSL}_2\R}

\newcommand{\dev}{\textsf{dev}}

 {\left\lbrace\begin{array}{@{}l@{}}}%
 {\end{array}\right.}

{\begin{list}{}
         {\setlength{\leftmargin}{#1}}
         \item[]
}
{\end{list}}

\def\qr#1#2{%
      \raise1ex\hbox{$#1$}\Big/ \lower1ex\hbox{$#2$}%
}

\def\qrr#1#2{%
      \raise1ex\hbox{$#1$}\Big/\Big/ \lower1ex\hbox{$#2$}%
}

\def\ql#1#2{%
      \lower1ex\hbox{$#1$}\Big\backslash \raise1ex\hbox{$#2$}%
}

 {\left\lbrace\begin{array}{@{}l@{}}}%
 {\end{array}\right.}

\begin{document}
\title[Distances on the Moduli space of $\cp-$Structures]{DISTANCES ON THE MODULI SPACE OF COMPLEX PROJECTIVE STRUCTURES}
\author{GIANLUCA FARACO}
\address{School of Mathematics - Tata Institute of Fundamental Research,	Homi Bhabha Road, Mumbai 400005, India}
\curraddr{}
\email{faraco@math.tifr.res.in}
\email{gianluca.faraco.math@gmail.com}

\thanks{}

\date{December 2018}

\begin{abstract}
Let $S$ be a closed and oriented surface of genus $g$ at least $2$. In this (mostly expository) article, the object of study is the space $\mathcal{P}(S)$ of marked isomorphism classes of projective structures on $S$. We show that $\mathcal{P}(S)$, endowed with the canonical complex structure, carries
exotic hermitian structures that extend the classical ones on the Teichm\"uller space $\mathcal{T}(S)$ of $S$. We shall notice also that the Kobayashi and Carath\'eodory pseudodistances, which can be defined for any complex manifold, can not be upgraded to a distance. We finally show that $\mathcal{P}(S)$ does not carry any Bergman pseudometric.
\end{abstract}

\maketitle
\tableofcontents

\section{Introduction}  

\subsection{About the problem} Teichm\"uller theory is one of those topics which were intensively studied in the last century. For a closed surface $S$, the Teichm\"uller space $\mathcal{T}(S)$ of $S$ is defined as the moduli space of deformations of complex structures defined on $S$. If $S$ has negative Euler characteristic, then the Teichm\"uller space carries a natural complex structure that makes it a complex manifold of dimension $3g-3$, where $g$ denotes the genus of $S$. The complex structure on $\mathcal{T}(S)$ can be defined in different ways and here we will consider the one introduced by Bers (the curious reader can consult \cite{SN} for another proof). More precisely, in the series of papers \cite{BE3,BE, BE2}, Bers 
defined an embedding, currently known as Bers' embedding, that realises the Teichm\"uller space as a bounded pseudoconvex domain inside the complex vector space $\C^{3g-3}$. Some results by Royden \cite{RO}, Oka \cite{O9,O4} and Kobayashi \cite{SK} combined togheter imply that the Teichm\"uller space with its natural complex structure is more than a complex manifold, indeed it is a Stein manifold. At the same time, the Teichm\"uller space carries different metrics, namely: the Teichm\"uller metric, the Bergman metric, the Weil-Petersson metric, the K\"ahler-Einstein metric, the McMullen metric (which is K\"ahler-hyperbolic), the Kobayashi metric and the Carath\'eodory metric. Each one arises from a particular viewpoint on the study of such space. With the only exception of the Teichm\"uller metric, we will summarise briefly these metrics in sections \ref{kcm}, \ref{km} and \ref{bpd}.\\
Upgrading a complex structure on $S$ to a complex projective structure by introducing a projective atlas makes it a rigider object but richer from the geometric viewpoint. For a closed surface $S$ of genus at least $2$, the moduli space of complex projective structures $\mathcal{P}(S)$ is defined in the same fashion of the Teichm\"uller space, namely as the space of deformations of projective structures on $S$. Any complex projective structure on $S$ induces an underlying complex structure: Indeed the major interests for this type of structures arise from the study of linear ODEs (see \cite{GKM} for instance) as well as classical uniformization theory (see \cite{GU} for instance). This fact leads to define a natural and continuous forgetful map that associates any projective structure its underlying complex structure. It can be shown that the forgetful map is actually a fibration over the Teichm\"uller space. For any given surface of genus $g\ge2$, the moduli space of complex projective structures on $S$ carries a natural complex structure that makes the forgetful map a holomorphic fibration.\\
\noindent In this work we investigate which metrics are naturally carried by $\mathcal{P}(S)$ endowed with its natural complex structure. We shall show the existence of exotic metrics that extend the classical ones carried by the Teichm\"uller space $\mathcal{T}(S)$ of $S$. More precisely: Denoting by $h_\bullet$ one amoung the Weil-Petersson, Bergman, K\"ahler-Einstein and McMullen metric, we shall prove the following result.\\

\noindent \textbf{Theorem \ref{mt}:} \emph{Let $S$ be a closed surface of genus $g\ge2$, and let $\mathcal{P}(S)$ be the moduli space of complex projective structure on $S$ endowed with the natural complex structure. Then there exists a hermitian metric on $\mathcal{P}(S)$ that extends the metric $h_\bullet$ on $\mathcal{T}(S)$. In particular this metrics turns out K\"ahler complete unless $h_\bullet$ is the Weil-Petersson metric.}\\

\noindent It would be interesting to understand if the moduli space $\mathcal{P}(S)$ carries a K\"ahler-Einstein metric such that its restriction on the Teichm\"uller space coincides with the K\"ahler-Einstein metric on $\mathcal{T}(S)$. Similarly, we may wonder if there exists a K\"ahler-hyperbolic metric that extends the McMullen metric on $\mathcal{T}(S)$. We shall discuss these problems in section \ref{dist}, however we anticipate that these questions are essentially open.\\
\noindent The moduli space $\mathcal{P}(S)$ endowed with its complex structure carries also the Kobayashi and Carath\'eodory pseudodistances which are classically defined on any complex manifold. Despite they are honest metrics on $\mathcal{T}(S)$ (endowed with the canonical complex structure), in the case of $\mathcal{P}(S)$ we have the following surprising result.\\

\noindent \textbf{Theorem \ref{kcnotcom}:} \emph{Both Kobayashi and Carath\'eodory pseudodistances on $\mathcal{P}(S)$ can not be upgraded to a distance.}\\

\noindent Finally, since the Teichm\"uller space is a Stein manifold, the moduli space $\mathcal{P}(S)$ can be realised also as an unbounded domain inside $\C^{6g-6}$. Hence it might carry a Bergman pseudometric which can be defined if the Hilbert space of square integrable holomorphic functions on that domain is big enough. In the case of $\mathcal{P}(S)$, we shall prove here the following result.\\

\noindent \textbf{Theorem \ref{bnotcom}:} \emph{The moduli space $\mathcal{P}(S)$ does not carry a Bergman metric.}\\

\noindent The paper is organised as follows. In section \ref{cs} we introduce the notions on complex and Stein manifolds we need along the paper. In section \ref{teich} we introduce the Teichm\"uller space for a given closed surface $S$ of genus at least $2$. The first subsections concern about the complex structure of $\mathcal{T}(S)$ whereas in the last subsections we summarise those metrics which are carried by $\mathcal{T}(S)$ endowed with its natural complex structure. Section \ref{cps} starts with the definition of complex projective structure on a given surface and the definition of $\mathcal{P}(S)$. After these definitions, we then turn the attention on the natural complex structure on this moduli space and we emphazises its relationship with the  Teichm\"uller space. Finally, section \ref{dist} we state and prove the main result of this work adding some comments about those questions that do not find an answer in this work.

\subsection*{Acknowlegments} The author wish to thank his Ph.D. advisor Stefano Francaviglia for introducing him to the topic of complex projective structure and his inner Ph.D. tutor Alberto Saracco for introducing him on Stein manifolds and metrics on complex spaces. The author also would like to 
thank Misha Kapovich for a useful remark on MathOverflow and Nicoletta Tardini for a useful conversation.\\

\section{Complex geometry and Stein manifolds} \label{cs}
\noindent In this section we give the main definitions and preliminaries on complex manifolds that we will use in the sequel. 

\begin{defn} Let $M$ be a topological manifold of real dimension $2n$. A \emph{complex structure} on $M$ is defined as the datum of a maximal complex atlas $\mathcal{A}=\{(U_\alpha,\varphi_\alpha)\}_{\alpha\in \Lambda}$ where $\{U_\alpha\}$ is an open cover of $M$ and any chart $\varphi_\alpha$ is a homeomorphism into its image $\varphi_\alpha(U_\alpha)\subset \C^n$ such that transition functions turn out to be biholomorphisms. 
\end{defn}

\noindent Since any biholomorphism is in particular a diffeomorphism, it turns out that any complex atlas defined on $M$ determines an underlying differentiable structure making $M$ a smooth manifold of dimension $2n$. Upgrading a topological manifold $M$ to a complex manifold by introducing a complex atlas, turns it into a more rigid object from many viewpoints and the sets of compatible functions are much smaller than their topological counterparts.

\begin{defn} Let $M$ be a complex manifold. A complex-valued function $f$ is said to be \emph{holomorphic} if the function $f\circ\varphi_\alpha^{-1}:\varphi_\alpha(U)\longrightarrow \C$ is holomorphic for any complex chart. We denote by $\mathcal{O}(M)$ the Fr\'echet algebra of all holomorphic functions on $M$ with the compact-open topology.
\end{defn}

\noindent For our purposes, we shall not need to recall other definition about complex manifolds, we then turn our attention to those complex manifolds known in literature as Stein manifolds. The curious reader can consult, for instance, the introductory book \cite{DH}.

\subsection{Domains of holomorphy} Before going to introduce Stein manifolds, we shall need to consider first the notion of \emph{domain of holomorphy} and its characterizations. Let us start with the following definition.

\begin{defn}\label{domhol}
Let $\Omega$ be an open subset of $\C^n$. Then $\Omega$ is called \emph{domain of holomorphy} if there exists a holomorphic function $f$ on $\Omega$ that is not holomorphically extendable to a larger domain. 
\end{defn}

\noindent In $1-$dimensional case, it is classical that for any open set $\Omega \subset \C$ there is a holomorphic function which is not holomorphically extendable over any boundary point of $\Omega$, this is known as analytic continuation property. This property is very specific to the dimension $1$ which is no longer true in dimension $n\ge2$. Indeed, Hartogs was the first to notice the existence of domains $\Omega$ in $\C^n$ (with $n\ge2$) on which every holomorphic function defined on $\Omega$ can be holomorphically extended to a larger domain in which $\Omega$ is contained (see \cite{HA}). Hartogs' discovery led to the search of the \emph{natural} domains of holomorphic functions, that is domains which are maximal in the sense that of definition \ref{domhol}. As we recall below, an important characterization of them was given by Cartan-Thullen in \cite{CT}. Let $M$ be a complex manifold. To any compact set $K$ of $M$ we can associate its $\mathcal{O}(M)$-\emph{hull} of $K$ which is defined as
\[ \widehat{K}_{\mathcal{O}(M)}=\Big\{ p \in M \text{ }\Big\vert\text{ } |f(p)|\leq\underset{x\in K}{\text{max}}|f(x)|, \text{ } \forall f \in\mathcal{O}(M)\Big\}.
\]

\begin{defn}
A compact set $K$ in a complex manifold $M$ is $\mathcal{O}(M)$-convex if $K= \widehat{K}_{\mathcal{O}(M)}$. A complex manifold $M$ is called \emph{holomorphically convex} if for every compact set $K$ in $M$ its $\mathcal{O}(M)$-\emph{hull} $\widehat{K}_{\mathcal{O}(M)}$ is also compact.
\end{defn}

\noindent The following theorem gives the characterization we have mentioned above. We refer to the excellent book \cite{FF} for the proof.

\begin{thm}[Cartan-Thullen \cite{CT}]\label{ct}
Let $\Omega$ be a domain in $\C^n$. Then $\Omega$ is a domain of holomorphy if and only if it is holomorphically convex.
\end{thm} 

\noindent It is natural to ask which geometric properties characterize domains of holomorphy. It happens that for any domain of holomorphy $\Omega$ in $\C^n$, every continuous mapping $f:\overline{\Bbb D}\times [0,1]\longrightarrow \C^n$ such that
\begin{enumerate}
\item for each $t\in[0,1]$ the mapping $f_t:\Bbb D\longrightarrow \C^n$ defined by $f_t(z)=f(z,t)$ is holomorphic, and
\item $f(z,t)\in \Omega$ unless $|z|<1$ and $t=1$, 
\end{enumerate}
maps $\overline{\Bbb D}\times [0,1]$ into $\Omega$, where $\Bbb D$ denotes the unit disc in $\C$. This geometric feature is known as \emph{pseudoconvexity} of $\Omega$. In \cite{O9, O4}, Oka showed that this property characterize domains of holomorphy completely. We have the following characterization theorem.

\begin{thm}[Oka \cite{O9,O4}]\label{ot}
Let $\Omega$ be a domain in $\C^n$. Then $\Omega$ is a domain of holomorphy if and only if it is pseudoconvex.
\end{thm}

\begin{rmk}
Our definition of pseudoconvexity was given by Oka in \cite{O4}. There are other definitions of pseudoconvexity known in literature as \emph{Hartogs pseudoconvexity} or \emph{Levi pseudoconvexity} for domains with $\mathcal{C}^2-$boundaries. As it is natural to expect, these definitions turn out equivalent. We refer to \cite{FF} and \cite{SK} and references therein for further details.
\end{rmk}

\subsection{Stein manifolds and the Oka-Grauert Principle} In this section we are going to consider a very special class of complex manifolds, namely Stein manifolds. The original definition of this class of manifolds was introduced by Stein in \cite{KS} by a system three axioms which postulate the existence of global holomorphic functions making an analogy with the properties of domains of holomorphy. Here we give the following modern definition.

\begin{defn}
A complex manifold $M$ is said to be \emph{Stein manifold} (or holomorphically complete manifold) if the following hold:
\begin{enumerate}
\item for every couple of distinct points $x\neq y$ in $M$ there is a holomorphic function $f\in\mathcal{O}(M)$ such that $f(x)\neq f(y)$;
\item $M$ is holomorphically convex.
\end{enumerate}
\end{defn}

\noindent Some remarks and considerations.

\SetLabelAlign{center}{\null\hfill\textbf{#1}\hfill\null}
\begin{enumerate}[leftmargin=1.75em, labelwidth=1.3em, align=center, itemsep=\parskip]
\item[\bf 1.] In \cite{RR}, Remmert gave the following characterization: A complex manifold $M$ is Stein if and only if it is biholomorphic to a closed complex submanifold of a Euclidean space $\C^N$ for some natural $N$. It follows that Stein manifolds are holomorphic analogues of affine algebraic manifolds.
\item[\bf 2.] From the Cartan-Thullen Theorem above \ref{ct}, it follows that an open set in $\C^n$ is Stein if and only if it is a domain of holomorphy.
\item[\bf 3.] Compact manifolds are never Stein. Infact, since any holomorphic function defined on a compact complex manifold is constant, it follows that there is no function that separetes any given pair of points.
\item[\bf 4.] If $\pi:E\longrightarrow M$ is a holomorphic vector bundle over a Stein base $M$, then the total space is Stein. However this miserably fails if $E$ is a fiber bundle over $M$, see \cite[Section 4.21]{FF} for further details.
\end{enumerate}

\noindent We finally turn our attention on vector bundles having a Stein base. The Oka - Grauert Principle asserts that every topological complex vector bundle over a Stein manifold admits an equivalent holomorphic vector bundle structure. In addition we have the following theorem.

\begin{thm}\label{topeqholo}
Let $M$ be a Stein manifold. Two holomorphic vector bundles over $M$ are holomorphically equivalent if and only if they are topologically equivalent. 
\end{thm}

\begin{cor}
Let $M$ be a simply connected Stein manifold. Then any holomorphic vector bundle over $M$ is holomorphically trivial.
\end{cor}

\noindent We refer to \cite[Section 5.3]{FF} and references therein for the proof of these results. Roughly speaking, theorem \ref{topeqholo} implies that the natural injection $V_{\textsf{holo}}^k(M)\hookrightarrow V_{\textsf{top}}^k(M)$ of the set of equivalence classes of holomorphic vector bundles over $M$ into the set of equivalence classes of topological vector bundles of rank $k$ is actually a bijection. 
\begin{rmk}
When $k=1$, the bijection above can be explained in cohomological terms. Let $\mathcal{C}^*$ (resp. $\mathcal{O}^*$) be the sheaf of continuous (resp. holomorphic) nonvanishing functions over $M$. If $M$ is a Stein manifold, then the homomorphism $H^1(M,\mathcal{O}^*)\longrightarrow H^1(M,\mathcal{C}^*)$, induced by the sheaf inclusion $\mathcal{O}^*\hookrightarrow \mathcal{C}^*$, turns out an isomorphism (see \cite{FF}). Since $V_{\textsf{holo}}^1(M)=H^1(M,\mathcal{O}^*)=\text{Pic}(M)$ and $V_{\textsf{top}}^1(M)=H^1(M,\mathcal{C}^*)$, we get the desired conclusion.
\end{rmk}

\section{Some Teichm\"uller Theory}\label{teich}

\noindent Let $S$ be a closed surface, that is compact without boundary. We always assume that $S$ is connected, oriented with genus at least $2$. 

\subsection{Basic Definitions} A \emph{Riemann surface} $X$ is a complex manifold of dimension one. Since we are assuming that $S$ has genus at least $2$, by Poincar\'e-Klein-Koebe Uniformization Theorem (see \cite[Chapter IV]{FK}), any Riemann surface $X$ on $S$ is of the form $\hyp/\Gamma$ where $\hyp$ is the hyperbolic plane and $\Gamma$ is a Fuchsian group, that is a discrete subgroup of $\pslr$ acting freely and properly discontinuously on $\hyp$. We shall refer to $\Gamma$ as \emph{Fuchsian model for} $X$. A marked complex structure is a couple $(X,f)$ where $X$ is a Riemann surface and $f:S\longrightarrow X$ is an orientation preserving diffeomorphism which is called \emph{marking}. Two marked complex structures $(X,f)$ and $(Y,g)$ are considered to be equivalent if there exists a biholomorphism $h:X\longrightarrow Y$ such that $g\circ h\circ f^{-1}:S\longrightarrow S$ is a diffeomorphism isotopic to the identity. The Teichm\"uller space $\mathcal{T}(S)$ of $S$ is defined as the set of marked complex structures on $S$ endowed with the compact-open topology. Notice that this is a honest construction of $\mathcal{T}(S)$ in the sense that it does not depend on the choice of any particular base point. A second construction of the Teichm\"uller space is possible by using orientation-preserving diffeomorphisms. Fix a closed Riemann surface $X$ and consider an arbitrary pair $(Y,f)$ of a closed Riemann surface $Y$ and an orientation-preserving diffeomorphism $f:X\longrightarrow Y$. Two pairs $(Y_1,f_1)$ and $(Y_2,f_2)$ are declared to be equivalent if the map $f_2\circ f_1^{-1}:Y_1\longrightarrow Y_2$ is homotopic to a biholomorphic mapping $h:Y_1\longrightarrow Y_2$. The set of all these equivalent classes is denoted by $\mathcal{T}(X)$. Unlike the previous one, the second construction depends on the choice of a base point. The spaces $\mathcal{T}(S)$ and $\mathcal{T}(X)$ can be identified with a bijective map that can be used to define a topology on the latter making it a topological space homeomorphic to the first one.

\subsection{Beltrami differentials and quadratic differentials} Let $f:\Omega\longrightarrow\Omega'$ be a homeomorphism between domains of $\C$. Then, $f$ is said to be \emph{quasiconformal} if it satisfies the Beltrami equation
\begin{equation}\label{be}
 \frac{\partial f}{\partial \overline{z}}=\mu(z)\frac{\partial f}{\partial z}
\end{equation} 
for some complex measurable function $\mu\in L^\infty(\Omega)$ such that $||\mu(z)||_\infty<1$. The function $\mu$ is called \emph{Beltrami coefficient of}$f$ on $\Omega$. 

\begin{thm}\label{qcmh}
Let $\mu$ be an arbitrary element of $L^\infty(\hyp)$ with $||\mu(z)||_\infty<1$. Then there exists a quasiconformal mapping $f:\hyp\longrightarrow \hyp$ having $\mu$ as Beltrami coefficient. Such a mapping $f$ can be extended to a homeomorphism of $\overline{\hyp}$ and is uniquely determinated by the normalization condition $f(0)=0$, $f(1)=1$ and $f(\infty)=\infty$. 
\end{thm}

\noindent \emph{Why are we interested in Beltrami differentials?} Let $X$ be a fixed Riemann surface. For any point $[Y,f]\in \mathcal{T}(X)$, we would like to compare the structure of $X$ with respect to the structure of $Y$. Of course $X$ and $Y$ turn out to be the same structure in $\mathcal{T}(X)$ if the diffeomorphism $f$ is actually a biholomorphism. The failure of $f$ from conformality is measured by a \emph{Beltrami differential}, namely a $L^\infty-$section of the complex line bundle $\xi=k^{-1}\otimes \overline{k}$, where $k$ denotes the canonical bundle on $X$ (that is the holomorphic cotangent bundle of $X$). In local coordinates, this $(-1,1)-$differential is usually denoted as
\[ \mu=\mu(z)\frac{d\overline{z}}{dz}
\] where $\mu(z)$ is a function in $L^\infty(X)$ such that $||\mu(z)||_\infty<1$. In local charts, $f$ looks like a map $\Omega\longrightarrow \Omega'$, where $\Omega$ and $\Omega'$ are domains of $\C$, satisfying the equation \ref{be} with Beltrami coefficient $\mu(z)$. Denote by $B(X)$ the Banach space of Beltrami differentials on $X$ endowed with the essential supremum norm and by $B(X)_1$ the open ball of those differentials with norm bounded by $1$ from above.\\
\noindent A \emph{holomorphic quadratic differential} on $X$ is defined as a holomorphic section of the complex line bundle $k^2$ and it is usually denoted as $q(z)dz^2$ in local coordinates. For any fixed Riemann surface $X$, let $Q(X)$ denote the complex Banach space of quadratic differential on $X$. If $X$ is of finite type, in particular if $X$ is closed, the Riemann-Roch theorem provides that the dimension of $Q(X)$ is $3g-3$, where $g$ denotes the genus of $X$. \\
\noindent For any Beltrami differential $\mu\in B(X)$ and any quadratic differential $q\in Q(X)$, the quantity given by the integral
\begin{equation}
 \int_X \mu q 
\end{equation} is well-defined. Indeed, the product of $\mu$ with $q$ defines a $(1,1)-$form over $X$ which can be integrated over $X$. Integration defines then a natural (but singular!) pairing  between these space.

\subsection{The Bers' surjection} Let $X$ be a fixed Riemann surface and let $\Gamma$ be the Fuchsian model of $X$. We assume, without loss of generality, that each of $0,1,\infty$ is fixed by some non-trivial element of $\Gamma$. Any quasiconformal map $f:X\longrightarrow Y$ onto a Riemann surface $Y$, with Beltrami coefficient $\mu_f$, lifts to a quasiconformal map $\widetilde{f}:\hyp\longrightarrow \hyp$ satisfying the Beltrami equation for a certain complex measurable function $\mu$ depending on $\mu_f$. Amoung all possible lifts $\widetilde{f}:\hyp\longrightarrow \hyp$ of $f$ there exists a preferred one, namely the lift that fixes each of $0,1,\infty$. Unless we state otherwise, from now on we always consider the preferred lift. Notice that such a lift exists and it is uniquely determinated by Theorem \ref{qcmh}. The preferred lift induces an injective homomorphism $\phi_f:\Gamma\longrightarrow \pslr$ which is defined by $\phi_f(\gamma)=\widetilde{f}\circ\gamma\circ {\widetilde{f}}^{-1}$, where $\gamma\in\Gamma$. Since $\phi_f$ is actually an isomorphism onto its image, it follows that $\phi_f(\Gamma)$ is a Fuchsian group $\Gamma^\mu$, namely the Fuchsian group that uniformize $Y$, \emph{i.e.} $\hyp/\Gamma^\mu=Y$. Consider now the Beltrami coefficient $\mu$ of the preferred lift of $f$. A straightforward computation shows that $\mu$ has an invariant property with respect to the action of $\Gamma$ on $\hyp$, namely
\[ \mu=\big(\mu\circ\gamma\big)\frac{\overline{\gamma'}}{\gamma'} \quad \text{ a.e. on } \hyp, \quad  \text{ where } \gamma\in\Gamma.
\] In this case, the coefficient $\mu$ is said to be Beltrami coefficient of $\hyp$ with respect to $\Gamma$. 

\begin{defn}
Set $B(\hyp,\Gamma)_1$ the open unit ball in the complex Banach space of all Beltrami coefficient $\mu\in L^\infty(\hyp)$ which are invariant with respect to the action of $\Gamma$.
\end{defn}

\noindent Any coefficient $\mu\in B(\hyp,\Gamma)_1$ determines a quasiconformal map $w^\mu:\hyp\longrightarrow \hyp$ having Beltrami coefficient $\mu$ by Theorem \ref{qcmh}. Such a map turns out $(\Gamma, \Gamma^\mu)-$equivariant, with $\Gamma^\mu=w^\mu\Gamma{w^\mu}^{-1}$, and descends to a quasiconformal map $w:X\longrightarrow Y$, with $Y=\hyp/\Gamma^\mu$. It can be shown that two different elements $\mu,\nu\in B(\hyp,\Gamma)_1$ define the same Fuchsian group if and only if $w^\mu=w^\nu$ on $\R$ (see for instance \cite{IT}). This leads to define a equivalence relation $\sim$ on $B(\hyp,\Gamma)_1$ such that two coefficients $\mu$ and $\nu$ are related if and only if the induced quasiconformal maps $w^\mu$ and $w^\nu$ agree on $\R$. We have the following result whose proof can be found in \cite[Section 3.3.1]{SN}.

\begin{prop}\label{bp}
The mapping $\beta$ that associates any coefficient $\mu$ its quasiconformal map $w^\mu$ defines a continuous surjection $\beta: B(\hyp,\Gamma)_1\longrightarrow \mathcal{T}(X)$ called Bers surjection. Moreover, the quotient mapping $\widetilde{\beta}:B(\hyp,\Gamma)_1/\sim\longrightarrow \mathcal{T}(X)$ defines a continuous bijection.
\end{prop}

\subsection{Simultaneous uniformization}\label{simunif} As above, let $X$ be a fixed Riemann surface and let $\Gamma$ its Fuchsian model. We are now going to identify the Teichm\"uller space $\mathcal{T}(X)$ with the set of quasiconformal mapping of $\cp$ which are conformal on the lower half-plane $\Bbb L$.
Any coefficient $\mu\in B(\hyp,\Gamma)_1$ can be extend to a coefficient $\widetilde{\mu}$ defined on $\C$ in the following way. For any Beltrami coefficient $\mu$ on $\hyp$ with respect to $\Gamma$ we define a coefficient on $\C$ by setting
\[ \widetilde{\mu}(z)=
\begin{cases}
\mu(z), \quad z\in\hyp\\
\\
0, \quad \quad z\in \C\setminus\hyp .
\end{cases}
\] Notice that $||\widetilde{\mu}||<1$. The following theorem holds.

\begin{thm}\label{qcmc}
Let $\mu$ be an arbitrary element of $L^\infty(\C)$ with $||\mu(z)||_\infty<1$. Then there exists a quasiconformal mapping $f:\cp\longrightarrow \cp$ having $\mu$ as Beltrami coefficient. Such a mapping $f$ is uniquely determinated by the following normalization condition $f(0)=0$, $f(1)=1$ and $f(\infty)=\infty$. 
\end{thm}

\noindent We call this map, uniquely determinated by the normalization condition, \emph{canonical quasiconformal map} and we denote it as $w_\mu$ (notice that here $\mu$ is subscript). Also in this case, the canonical quasiconformal map $w_\mu$ induces an isomorphism $\psi_\mu(\gamma)=w_\mu\circ\gamma\circ w_\mu^{-1}$, where $\gamma\in\Gamma$. The mapping $w_\mu$ is a quasiconformal transformation of the Riemann sphere that might not preserve the upper half-plane $\hyp$, hence the image $\Gamma_\mu$ of $\psi$ is a subgroup of $\pslc$ called quasi-Fuchsian group. The group $\Gamma_\mu$ acts freely and properly discontinuously on both $w_\mu(\hyp)$ and $w_\mu(\Bbb L)$, hence it uniformizes two Riemann surfaces simultaneously. More precisely, the quasiconformal mapping $w_\mu$ defines a quasiconformal mapping of $X$ to $w_\mu(\hyp)/\Gamma_\mu$ and a biholomorphism of $X^*$ to $w_\mu(\Bbb L)/\Gamma_\mu$, where $X^*$ is the mirror image of $X$. This is know in literature as \emph{Bers' simultaneous uniformization}, see also \cite{BE}.\\
\noindent If the coefficients $\mu,\nu\in B(H,\Gamma)_1$ define the canonical quasiconformal maps $w_\mu$ and $w_\nu$ respectively, then $w_\mu=w_\nu$ on $\Bbb L$ if and only if the quasiconformal maps $w^\mu$ and $w^\nu$ coincide on $\R$ (see \cite[Lemma 6.1]{IT}). This fact leads to declare two canonical maps $w_\mu$ and $w_\nu$ equivalent if $w_\mu=w_\nu$ on $\Bbb L$ and it easy to see that the set of equivalence classes turns out in bijective correspondence with the Teichm\"uller space $\mathcal{T}(X)$. We set
\[ \mathcal{T}(\Gamma)=\Big\{ [w^\mu] \text{ }\big|\text{ } w^\mu \text{ is a quasiconformal mapping of } \cp \text{ such that } \Gamma_\mu \text{ is a quasi-Fuchsian group} \Big\}
\] The topology on $\mathcal{T}(\Gamma)$ is induced from that of $\mathcal{T}(X)$ (and hence from that of $\mathcal{T}(S)$), hence they can be identified as topological spaces. The space $\mathcal{T}(\Gamma)$ is known as Teichm\"uller space of $\Gamma$ or deformation space of $\Gamma$.

\subsection{Schwarzian derivative}\label{sd} Let $\Omega\subset \C$ be an connected domain. The \emph{Schwarzian derivative} of a locally injective holomorphic map $f:\Omega\longrightarrow \cp$, is the holomorphic quadratic differential defined as
\[ \mathcal{S}(f)=\Bigg[\Bigg(\frac{f''(z)}{f'(z)}\Bigg)'-\frac{1}{2}\Bigg(\frac{f''(z)}{f'(z)}\Bigg)^2\Bigg]dz^2
\] Intuitively, the quadratic differential $\mathcal{S}(f)$ measures the failure of $f$ to be a M\"obius transformation. The Schwarizian derivative satisfies two properties, namely
\SetLabelAlign{center}{\null\hfill\textbf{#1}\hfill\null}
\begin{enumerate}[leftmargin=1.75em, labelwidth=1.3em, align=center, itemsep=\parskip]
\item[\bf 1.] For any $g\in\pslc$, we have 
\[ \mathcal{S}(g)\equiv 0, 
\] and
\item[\bf 2.] Cocycle property, if $f,g$ are locally injective holomorphic maps such the composition is well-defined, then
\[ \mathcal{S}(f\circ g)=g'(z)^2\mathcal{S}(f)(g(z))+\mathcal{S}(g)=g^*\mathcal{S}(f)+\mathcal{S}(g).
\] 
\end{enumerate}

\noindent Notice that any map $f$ is almost determinated by its Schwarzian derivative, indeed if $g$ is another function such that $\mathcal{S}(f)=\mathcal{S}(g)$, then $f$ and $g$ differ by some M\"obius transformation.

\subsection{The Bers' embedding and the complex structure on $\mathcal{T}(S)$} In this section we recall the natural complex structure on $\mathcal{T}(S)$ of a closed surface $S$. Following Bers, we shall realize the Teichm\"uller space as a bounded domain in $\C^{3g-3}$, where $g$ denotes as usual the genus of $S$. In order to do this, let $X$ be a fixed Riemann surface and let $\Gamma$ be its Fuchsian model. We consider the Teich\"uller space of $\Gamma$ which we know to be homeomorphic to $\mathcal{T}(S)$. For any element $\mu \in B(\hyp,\Gamma)_1$, we set
\[ \varphi_\mu(z)=\mathcal{S}(w_\mu)(z) \qquad z\in \Bbb L.
\] It can be shown that the quadratic differential $\varphi_\mu(z)$ satisfies an equivariant property in the following sense
\[ \varphi_\mu\big(\gamma(z)\big)=\varphi_\mu(z) \qquad \text{ for any }\gamma\in\Gamma \quad \text{ and } z\in \Bbb L.
\] In other words, $\varphi_\mu$ is a quadratic differential on $\Bbb L$ with respect to $\Gamma$. In particular, by the equivariant property, $\varphi_\mu$ descends to a quadratic differential $q(z)dz^2\in Q(\Bbb L/\Gamma)$. 

\begin{defn}\label{defqd}
Let $Q(\Bbb L,\Gamma)$ denote the complex Banach space of holomorphic quadratic differentials on $\Bbb L$ with respect $\Gamma$ equipped with the norm defined as follows
\[ ||\varphi||_\infty=\sup_{\Bbb L} \big( \text{Im } z\big)^2|\varphi(z)|.
\]
\end{defn}

\noindent The Schwarzian derivative defines a function
\[ 
\begin{aligned}
\Phi: B(\hyp,\Gamma&)_1& \longrightarrow& &Q(\Bbb L,\Gamma)\\
\mu &  &\longmapsto & & \varphi_\mu(z)\\
\end{aligned}
\] called \emph{Bers' projection}. If two coefficient $\mu$ and $\nu$ define the canonical quasiconformal maps $w_\mu$ and $w_\nu$ respectively, then $[w_\mu]=[w_\nu]\in\mathcal{T}(\Gamma)$ if and only if $\varphi_\mu=\varphi_\nu$ on $\Bbb L$ (see \cite[Lemma 6.4]{IT}). Hence $\Phi$ descends to an injective function
\[ 
\begin{aligned}
\mathcal{B}:\mathcal{T}(\Gamma)&\longrightarrow &Q(\Bbb L,\Gamma)\\
[w_\mu] & \longmapsto & \varphi_\mu(z)\\
\end{aligned}
\] called \emph{Bers' embedding} such that 
\[ \Phi(\mu)=\mathcal{B}\circ \beta,
\] where $\beta$ is the Bers' surjection defined before. We have the following proposition whose proof can be found for instance in \cite[Proposition 6.5]{IT}.

\begin{prop}
Both Bers' projection $\Phi: B(\hyp,\Gamma)_1 \to Q(\Bbb L,\Gamma)$ and Bers' embedding $\mathcal{B}:\mathcal{T}(\Gamma) \to Q(\Bbb L,\Gamma)$ are continuous.
\end{prop}

\noindent From the topological point of view, it can be shown that the Teichm\"uller space $\mathcal{T}(S)$ is homeomorphic to $\R^{6g-6}$ (see \cite[Theorem 5.15]{IT}, for instance), and so are $\mathcal{T}(X)$ and $\mathcal{T}(\Gamma)$. Brower's theorem of invariance of domains implies that the image of the continuous injection $\mathcal{B}$ is a domain inside $Q(\Bbb L,\Gamma)$, hence an homeomorphism onto its image. The Banach space $Q(\Bbb L,\Gamma)$ is a complex vector space of dimension $3g-3$ by Riemann-Roch Theorem, thus $\mathcal{T}(\Gamma)$ inherits the complex structure of $\C^{3g-3}$, and so $\mathcal{T}(X)$ and $\mathcal{T}(S)$. Furthermore, the image of $\mathcal{B}$ lies inside the open ball in $Q(\Bbb L,\Gamma)$ of center $0$ and radius $3/2$ with respect the norm defined in \ref{defqd}. This turns out to be a consequence of Nehari-Kraus' theorem that states that every univalent function $f$ on $\Bbb L$ (and then any conformal maps $w_\mu$ with $\mu\in B(\hyp,\Gamma)_1$) satisfies the inequality
\[ ||\mathcal{S}(f)||_\infty=\sup_{\Bbb L} \big( \text{Im } z\big)^2\big|\mathcal{S}(f)\big|\le\frac32.
\]

\noindent Some remarks.

\SetLabelAlign{center}{\null\hfill\textbf{#1}\hfill\null}
\begin{enumerate}[leftmargin=1.75em, labelwidth=1.3em, align=center, itemsep=\parskip]
\item[\bf 1.] It can be shown that the complex structure in $\mathcal{T}(S)$ does not depend on the choice of the Riemann surface $X$ and then on the choice of the Fuchsian model $\Gamma$. Let $Y$ be another Riemann surface different to $X$ with Fuchsian model $\Lambda$, and define $\omega$ as a lift of the quasiconformal map $f:X\longrightarrow Y$. The mapping $\omega$ induce an homeomorphism 
\[ \omega_*:\mathcal{T}(\Gamma)\longrightarrow \mathcal{T}(\Lambda)
\] which can be taught as a translation of the base point. In \cite{IT}, the Authors show that such a map is actually a biholomorphism, thus the complex structure on $\mathcal{T}(S)$ does not depend on the choice of the base point.
\item[\bf 2.] The reader might be unhappy on this definition of complex structure on $\mathcal{T}(S)$ since some choices are required. Another way of defining the complex structure on $\mathcal{T}(S)$ come from Kodaira-Spencer deformation theory where no choice is needed. In this case, the Teichm\"uller space is realized as an abstract complex manifold of dimension $3g-3$. 
\end{enumerate}

\subsection{The Kobayashi and Carath\'eodory distances on $\mathcal{T}(S)$}\label{kcm} We shall now introduce two distances on the Teichm\"uller space $\mathcal{T}(S)$ known as \emph{Kobayashi distance} and \emph{Carath\'eodory distance}. Let us introduce them in full generality.\\
\noindent Let $M$ be a complex manifold. We start defining a pseudodistance $k_M$ on $M$ as follows. For any pair of given points $p,q\in M$, we choose a $k-$string between $p,q$, that is a finite sequence of points $p=p_o,\dots,p_k=q$, points $a_1,\dots,a_k,b_1,\dots,b_k$ and holomorphic mapping $f_1,\dots,f_k$ of $\Bbb D$ into $M$ such that $f_i(a_i)=p_{i-1}$ and $f_i(b_i)=p_i$. For each choice of strings, points and mappings we consider the quantity given by
\[ \sum_{i=1}^k d_{\Bbb D}(a_i,b_i).
\] The Kobayashi pseudodistance $k_M$ on $M$ is defined to be the infimum of those numbers obtained in this manner. In a similar fashion, the Carath\'eodory pseudodistance $c_M$ on a complex manifold $M$ is defined by
\[ c_M(p,q)=\sup_{f} d_{\Bbb D}\big(f(p),f(q)\big) \quad \text{ for } p,q\in M,
\] where the supremum is taken with respect to the family of holomorphic mappings $f:M\longrightarrow \Bbb D$. It is an easy matter to check that both $k_M$ and $c_M$ are continuous and satisfy the axioms of pseudodistance. It may very well happens that both pseudodistances cannot be upgraded to a distance in $M$; for instance if $M=\C^n$ then they satisfy the equality $k_M\equiv c_M\equiv 0$. The Kobayashi and Carath\'eodory pseudodistances are related by the following inequality.

\begin{prop}\label{compkc}
Let $M$ be a complex manifold. For any $p,q\in M$ we have $k_M(p,q)\ge c_M(p,q)$.
\end{prop}

\noindent The proof of this proposition can be found in \cite[Proposition 2.1]{SK}. An immediate consequence of the previous result is that the Kobayashi pseudodistance $k_M$ is actually a distance as soon as the $c_M$ is a distance on $M$. Since Carath\'eodory pseudodistance between two points in $M$ is given by taking the supremum of all holomorphic mappings $f:M\longrightarrow \Bbb D$, a necessary and sufficient condition for which $c_M$ is a distance is that the family of such mappings separete the points in $M$.

\begin{defn}
Let $M$ be a complex manifold and let $k_M$ be the Kobayashi pseudodistance on $M$. If $k_M$ is a distance, then $M$ is called \emph{hyperbolic manifold}. If $k_M$ is also complete, then $M$ is called \emph{complete hyperbolic manifold}.
\end{defn}

\begin{thm}\label{tehm}
The Teichm\"uller space $\mathcal{T}(S)$ of $S$ is a complete hyperbolic manifold.
\end{thm}

\noindent This theorem was proven indipendently by Royden in \cite{RO} and Earle-Kra in \cite{EK}. The proof of this theorem can be found also in \cite[Theorem 6.21]{IT}. Roughly speaking, the Kobayashi distance agrees with the Teichm\"uller distance, and since the latter is complete \cite[Theorem 5.4]{IT}, also the Kobayshi distance is complete. As a consequence we have the following result.

\begin{thm}\label{teiss}
The Teichm\"uller space $\mathcal{T}(S)$ of $S$ is a Stein manifold.
\end{thm}

\noindent Before giving the proof of this theorem, we shall need of the following result whose proof can be found in \cite[Theorem 3.4]{SK}.

\begin{lem}\label{pslem}
If a domain $M$ in $\C^n$ is hyperbolic complete, then it is pseudoconvex.
\end{lem}

\begin{proof}[Proof of Theorem \ref{teiss}]
Let $X$ be a Riemann surface and let $\Gamma$ its Fuchsian model. Bers' embedding realizes $\mathcal{T}(\Gamma)$ as a bounded domain inside $Q(\Bbb L,\Gamma)\cong \C^{3g-3}$ from which inherits its complex structure. Since $\mathcal{T}(\Gamma)$ is identified with $\mathcal{T}(S)$, the latter can be also taught as a bounded domain in $\C^{3g-3}$. By Theorem \ref{tehm}, the Teichm\"uller space $\mathcal{T}(S)$ of $S$ is a complete hyperbolic manifold, thus is pseudoconvex by Lemma \ref{pslem}. By Oka's theorem \ref{ot}, $\mathcal{T}(S)$ is a domain of holomorphy, thus holomorphically convex by Cartan-Thullen Theorem \ref{ct}. Hence the Teichm\"uller space $\mathcal{T}(S)$ is a Stein manifold.
\end{proof}

\noindent Some final considerations. The Carath\'eodory pseudodistance $c_{\mathcal{T}(S)}$ on $\mathcal{T}(S)$ is a distance. In \cite{EC}, Earle has shown that the $c_{\mathcal{T}(S)}$ is complete on $\mathcal{T}(S)$ and proportional to the Kobayashi distance $k_{\mathcal{T}(S)}$. In \cite{KI}, Kra has studied the connection between the Carath\'eodory distance with the Kobayashi distance showing that they agree on Abelian Teichm\"uller discs in $\mathcal{T}(S)$. This fact led to conjecture that these distances agree on the whole space, but it was shown very recently that this is not the case. This longstanding problem was solved by Markovic in \cite{MV}. Finally, Theorem \ref{teiss} can be derived also from a result by Horstmann in \cite{HH}. Indeed, he has shown that any domain in $\C^n$ which is complete with respect to the Carath\'eodory metric is holomorphically convex.

\subsection{K\"ahlerian metrics on $\mathcal{T}(S)$} \label{km} In this section we briefly recall some facts of those K\"ahler metrics on $\mathcal{T}(S)$ coming from the complex structure of the Teichm\"uller space. 

\SetLabelAlign{center}{\null\hfill\textbf{#1}\hfill\null}
\begin{enumerate}[leftmargin=1.75em, labelwidth=1.3em, align=center, itemsep=\parskip]
\item[\bf 1.] \textbf{Bergman metric:} Bers' embedding realises the Teichmüller space as a domain of holomorphy and hence it also carries a Bergman metric $b_{\mathcal{T}(S)}$ which turns out K\"ahler complete on $\mathcal{T}(S)$. This result is essentially due to Earle and Hahn. Indeed, the latter has proved that the Carath\'eodory metric on a bounded domain of $\C^n$ is bounded from above by the Bergman metric. Instead, Earle has proved that the Carath\'eodory metric $c_{\mathcal{T}(S)}$ on $\mathcal{T}(S)$ is complete. Their results combined togheter imply that $b_{\mathcal{T}(S)}$ is complete. A more recent proof is given by Chen in \cite{CBY}, who proved that the distance induced by $b_{\mathcal{T}(S)}$ is equivalent to the Teichm\"uller distance on $\mathcal{T}(S)$, and the latter is complete, providing a bi-Lipschitz model for $b_{\mathcal{T}(S)}$. We will describe the Bergman distance in more details in \ref{bpd}.\\
\item[\bf 2.] \textbf{Weil-Petersson metric:} A more important metric on the Teichmüller space is given the \emph{Weil-Petersson metric} $h_{WP}$. It can be shown that the cotangent space of $\mathcal{T}(S)$ at any point $X$ is given by the complex Banach space $Q(X)$ of quadratic differentials over $X$. For any pair of quadratic differentials $q_1(z)dz^2$ and $q_2(z) dz^2$ on $X$, the Weil-Petersson product on $Q(X)$ turns out a Hermitian product which is defined as follow: For any $q_1,q_2\in Q(X)$ we set
\[ \langle q_1,q_2\rangle_{WP}=\int_X \big(\text{Im } z\big)^2q_1(z)\overline{q_2(z)}dzd\overline{z}.
\] By duality, this product defines a Hermitian product (also denoted by $\langle \cdot,\cdot\rangle_{WP}$ with abuse of notation) on the tangent space of $\mathcal{T}(S)$ at any point $X$, hence a Hermitian metric $h_{WP}$ on $\mathcal{T}(S)$. Ahlfors in \cite{AL} showed that the Weil-Petersson metric is K\"ahlerian but not complete. Therefore $b_{\mathcal{T}(S)}$ is not equivalent to $h_{WP}$ on $\mathcal{T}(S)$.\\
\item[\bf 3.] \textbf{K\"ahler-Einstein metric:} A Riemannian metric $h$ on a complex manifold is called \emph{Einstein metric} if the Ricci tensor is proportional to the metric; that is the following equation holds: $\text{Ric}=k\cdot h$ for some constant $k$. A K\"ahler–Einstein metric is a Riemannian metric which is both a K\"ahler metric and Einstein metric. Cheng and Yau showed in \cite{CY} the existence of a unique complete K\"ahler–Einstein metric on any bounded domain of the complex space $\C^n$. Since the Teichmüller space is a domain of $\C^{3g-3}$ by Bers' embedding and bounded by Nehari-Kraus' Theorem, it carries a K\"ahler–Einstein metric with constant negative scalar curvature.\\
\item[\bf 4.] \textbf{McMullen metric:} McMullen defined in \cite{McM} a complete K\"ahler metric on $\mathcal{T}(S)$ with bounded sectional curvature which is K\"ahler-hyperbolic. The notion of K\"ahler-hyperbolic manifold was first introduce by Gromov in \cite{GR}.
\end{enumerate}

\begin{rmk}
With the only exception of the Weil-Petersson metric, all K\"ahlerian metrics on $\mathcal{T}(S)$ are quasi-isometrics.
\end{rmk}

\section{Complex Projective Structures}\label{cps}

\noindent A \emph{complex projective structure} $\sigma$ on $S$ is a maximal atlas whose charts take values on the Riemann sphere $\cp$ and transition functions are restrictions of M\"obius transformations.  From now on, the word \emph{complex} will be a blanket assumption, and we refer to these structure only as \emph{projective structures}. Also, we shall treat a projective structure $\sigma$ on $S$ as a surface in its own right for semplicity.

\begin{rmk}
Projective structures can be defined also on surfaces of genus lower than $2$. The sphere has a unique projective structure coming from the identification $\Bbb S^2\cong \cp$ up to isotopy, whereas any projective structure on a torus come from an affine structure. In the sequel we continue to assume that $S$ is a closed surface, connected, oriented with genus at least $2$.
\end{rmk}

\noindent A marked projective structure is a couple $(\sigma,f)$ where $\sigma$ is a projective structure and $f$ is an orientation preserving diffeomorphism $f:S\longrightarrow \sigma$. Two marked structures $(\sigma_1,f)$ and $(\sigma_2,g)$ are considered to be equivalent if there exists a projective isomophism $h:\sigma_1\longrightarrow \sigma_2$ such that $g\circ h \circ f^{-1}:S\longrightarrow S$ is isotopic to the identity. We set $\mathcal{P}(S)$ the set of marked isomorphism classes of projective structures on $S$. 

\subsection{Making $\mathcal{P}(S)$ a topological space} We now describe how to put a topology on the set $\mathcal{P}(S)$. In terms of geometric structures, any projective structure can be seen as a $\big(\cp,\pslc\big)$-structure. Therefore any projective structure is the same as an equivalent class of development-holonomy pair $(\dev,\rho)$, where $\dev:\widetilde{S}\longrightarrow \cp$ is an orientation-preserving smooth map equivariant with respect to a representation $\rho:\pi_1S\longrightarrow \pslc$. Two such a pairs $(\dev_1,\rho_1)$ and $(\dev_2,\rho_2)$ are declared to be equivalent if there exists an element $g\in\pslc$ such that $\dev_2=g\circ\dev_1$ and $\rho_2=g\rho_1g^{-1}$. The set $\mathcal{P}(S)$ can be seen as the quotient space by the action of the group $\textsf{Diff}^+_0(S)$ on the set of equivalent classes developing-holonomy pairs. Giving to the set of developing-holonomy pairs the compact-open topology, the quotient space $\mathcal{P}(S)$ inherits the quotient topology. 

\subsection{Relationship between $\mathcal{P}(S)$ and $\mathcal{T}(S)$} Since M\"obius transformations are holomorphic mappings, any projective structure $\sigma$ on $S$ defines an underlying complex structure making $S$ a Riemann surface. Conversely, by the classical uniformization theory, any Riemann surface $X$ is of the form $\hyp/\Gamma$ where $\Gamma$ is a Fuchsian group. In particular, this endows $X$ with a complex projective structure, namely the one coming from the identification $X\cong \hyp/\Gamma$. 

\begin{defn}\label{def_unif}
Let $X$ be a Riemann surface of genus $g\ge2$. Then we call the \emph{Fuchsian uniformization} of $X$ the natural complex projective structure coming from the quotient $\hyp/\Gamma$.
\end{defn}

\begin{rmk}\label{natcp}
More generally, the natural projective structure on $\cp$ induces a natural projective structure on any open set $U\subset \cp$. If a group $\Gamma$ acts on $U$ freely and properly discontinuously, the quotient surface $U/\Gamma$ inherits a natural projective structure. On the other hand, not every projective structure is of the form $U/\Gamma$: For instance Maskit has produced many examples of projective structures with surjective and non injective developing maps, via a geometric construction known as grafting, which consists in replacing a simple closed curve by an annulus (see \cite{MB}).
\end{rmk}

\noindent If two marked projective structures $(\sigma_1,f)$ and $(\sigma_2,g)$ are related by some projective isomophism $h:\sigma_1\longrightarrow \sigma_2$, then it is an easy matter to check that the underlying Riemann surfaces are related by the same isomorphism. As a consequence, there is a continuous \emph{forgetful map}
\[ \pi:\mathcal{P}(S)\longrightarrow \mathcal{T}(S),
\] where $\mathcal{T}(S)$ is the Teichm\"uller space of $S$, that associates any class of marked projective structures to its class of marked Riemann surfaces. Since the fibre of any Riemann surface $X$ contains the Fuchsian uniformization of $X$, the mapping $\pi$ is surjective. On the other hand the forgetful map fails to be injective. This is mainly due to the fact that isomorphism of projective structures turns out a stronger condition than isomorphism of complex structures. Following Loustau in \cite{LB}, we are going to give a brief description of the fibres. A more constructive and elementary description of the fibre is given also in \cite{DU}.\\
\noindent Let $\mathcal{P}(X)$ denote the fibre of $X$, namely the subset of those marked projective structure having $X$ ad underlying complex structure. For any given pair of structures $\sigma_1$ and $\sigma_2$, the identity map $\text{id}_S:\sigma_1\longrightarrow \sigma_2$ is holomorphic isomorphism but not projective, unless $\sigma_1\simeq\sigma_2$. The Schwarzian derivative $\mathcal{S}(\text{id}_S)$ can be used to measure the failure of $\text{id}_S$ to be projective. Equivalently, the Schwarzian derivative measures the difference between the structures $\sigma_1$ and $\sigma_2$. It can be shown that for any projective structure $\sigma \in \mathcal{P}(X)$ and any quadratic differential $\varphi\in Q(X)$ there exists a projective structure $\sigma_\varphi$ such that $\mathcal{S}\big(\text{id}_S:\sigma\longrightarrow \sigma_\varphi\big)=\varphi$. This is mainly due to by the fact the complex Banach space $Q(X)$ acts freely and transitively on the set $\mathcal{P}(X)$. As a consequence $\mathcal{P}(X)$ is a complex affine space modeled on $Q(X)$. This is known in literature as \emph{Schwarzian parametrization of the fibres}. Any choice of a basepoint $\sigma_o$ gives a well-defined isomorphism $\mathcal{P}(X)\to Q(X)$ such that $\sigma\mapsto \mathcal{S}\big(\text{id}_S:\sigma_o\to\sigma\big)=\varphi$. In the sequel we shall denote $\mathcal{S}\big(\text{id}_S:\sigma_o\to\sigma\big)$ simply as $\sigma-\sigma_o$. Recalling that $Q(X)$ is identified with the cotangent space of $\mathcal{T}(S)$ at $X$, the space $\mathcal{P}(S)$ is an affine holomorphic bundle modeled on the holomorphic cotangent bundle $T^*\mathcal{T}(S)$.

\subsection{The canonical complex structure} In this section we are going to describe how the space $\mathcal{P}(S)$ can be upgraded to a complex manifold of dimension $6g-6$, where $g$ denotes as usual the genus of $S$. As a consequence of the previous section, the moduli space of projective structures $\mathcal{P}(S)$ can be identified with the cotangent bundle of the Teichm\"uller space $T^*\mathcal{T}(S)$ by choosing a zero-section $s:\mathcal{T}(S)\to\mathcal{P}(S)$. Let us be more precisely: Any zero-section $s$ yields an isomorphism of complex affine bundles by using the Schwarzian parametrization of the fibre in the following way
\[
\sigma  \longmapsto \Big(\pi(\sigma), \sigma-s\big(\pi(\sigma)\big)\Big)
\] 
\noindent The cotangent space $T^*\mathcal{T}(S)$ is a complex manifold of dimension $6g-6$ and its complex structure can pulled back to define a complex structure on $\mathcal{P}(S)$. Different sections $s_1$ and $s_2$ produce different complex structures which are actually the same if and only if the different $s_1-s_2$ is a holomorphic section of $T^*\mathcal{T}(S)$.\\

\noindent An important class of sections is given by \emph{Bers sections}. By Bers' simultaneous uniformization, given $X,Y\in \mathcal{T}(S)$ there exists a discrete subgroup $Q(X,Y)$ of $\pslc$ that uniformizes $X$ and $Y$ simultaneously.\\
\noindent In order to making a comparison with section \ref{simunif}, let $f:Y\longrightarrow X$ be a quasiconformal map and let $\mu\in B(\hyp,\Gamma)_1$ be the Beltrami coefficient of its preferred lift, where $\Gamma$ is the Fuchsian model of $Y$. The extension of $\mu$ to the whole complex plane defines a quasiconformal map $w_\mu$ and then a quasi-Fuchsian group $Q(X,Y)=\Gamma_\mu$ that uniformizes $X$ and $Y^*$ simultaneously (where $Y^*$ is the mirror image of $Y$ on the lower half-plane). The open set $w_\mu(\hyp)$ is invariant by the action of $Q(X,Y)$ which acts freely and properly discontinuously. By remark \ref{natcp}, the quotient surface $X\simeq w_\mu(\hyp)/Q(X,Y)$ inherits a natural projective structure which we denote by $s_Y(X)$. Notice that the underlying Riemann of $s_Y(X)$ is $X$, so for any fixed $Y\in \mathcal{T}(S)$ we can define the Bers section as
\[ s_Y:\mathcal{T}(S)\longrightarrow \mathcal{P}(S)
\]
\[ \qquad X\longmapsto s_Y(X)
\]
\noindent It can be shown that every Bers section induces the same complex structure on $\mathcal{P}(S)$. We will refer to that structure as  \emph{canonical complex structure} on $\mathcal{P}(S)$.

\section{Distances on $\mathcal{P}(S)$}\label{dist}

\noindent In this last section we are going to show the existence of exotic hermitian metrics on $\mathcal{P}(S)$ that extend the classical known metrics on $\mathcal{T}(S)$.

\subsection{Exotic metrics on $\mathcal{P}(S)$} We have seen in \ref{cps} that any section induces an identification between $\mathcal{P}(S)$ and the cotangent bundle $T^*\mathcal{T}(S)$ given by the Schwarzian parametrization. This identification can be used to transport the natural complex structure on $T^*\mathcal{T}(S)$ to the moduli space $\mathcal{P}(S)$, namely the latter is endowed with the unique complex structure that makes the identification a biholomorphism. Recall that any Bers' section $s_Y$, where $Y\in\mathcal{T}(S)$, induces the natural complex structure. The cotangent space $T^*\mathcal{T}(S)$ is a complex vector bundle of complex rank $3g-3$ over the Teichm\"uller space which is known to be a Stein manifold by \ref{teiss}. Since $\mathcal{T}(S)$ is contractible (it is indeed homeomorphic to $\R^{6g-6}$), any complex vector bundle is topologically trivial, that means the existence of a homeomorphism between $T^*\mathcal{T}(S)$ and $\mathcal{T}(S)\times \C^{3g-3}$. By Theorem \ref{topeqholo}; two vector bundles over a Stein base are holomorphically trivial if and only if they are topologically trivial. Hence $T^*\mathcal{T}(S)$ and $\mathcal{T}(S)\times \C^{3g-3}$ endowed with their canonical complex structures are actually biholomorphic. In particular, we can deduce that $\mathcal{P}(S)$ with its canonical complex structure is biholomorphic to $\mathcal{T}(S)\times \C^{3g-3}$. Let us denote by $h_\bullet$ one of the following metrics $\{ b_{\mathcal{T}(S)}, h_{WP}, h_{KE}, h_{MM}\}$. Then we have the following theorem.

\begin{thm}\label{mt}
Let $S$ be a closed surface of genus $g\ge2$, and let $\mathcal{P}(S)$ be the moduli space of complex projective structure on $S$ endowed with the natural complex structure. Then there exists a hermitian metric on $\mathcal{P}(S)$ that extends the metric $h_\bullet$ on $\mathcal{T}(S)$. In particular, this metrics turn out K\"ahler complete unless $h_\bullet$ is the Weil-Petersson metric.
\end{thm}

\begin{proof}[Proof of \ref{mt}]
Let $h_\bullet$ be one of the metrics we are considering on $\mathcal{T}(S)$. Let $h_{g}$ be any hermitian metric on $\C^{3g-3}$. The product metric $h_\bullet\times h_g$ defines an hermitian metric on $\mathcal{T}(S)\times \C^{3g-3}$ which can be transported to a hermitian metric on $T^*\mathcal{T}(S)$ and then on $\mathcal{P}(S)$ via the identification given by the Schwarzian parametrization. Since both $h_\bullet$ and $h_g$ are K\"ahler, the product turns out K\"ahler. Finally, since the Weil-Petersson metric on $\mathcal{T}(S)$ is not complete we have that $h_{WP}\times h_g$ is not K\"ahler-complete. In all other cases, the metric $h_\bullet$ is K\"ahler-complete, hence the product metric is K\"ahler-complete.
\end{proof}

\noindent The following corollary is a straighforward consequence of the previous theorems.

\begin{cor}
The Weil-Petersson metric is not equivalent to any other metrics defined on $\mathcal{P}(S)$.
\end{cor}

\noindent Some comments. The metrics we have defined on $\mathcal{P}(S)$ do not preserve in general the type of metric defined on the $\mathcal{T}(S)$. For instance, there exist complete metrics on $\mathcal{P}(S)$ that extend the complete Bergman metric $b_{\mathcal{T}(S)}$ on $\mathcal{T}(S)$ without being a Bergman metric on $\mathcal{P}(S)$ in general: Indeed if a Bergman metric exists it will be unique. We will back on the Bergman pseudodistance on $\mathcal{P}(S)$ in the last paragraph \ref{bpd}.\\
\noindent A similar discussion can be made for those metrics that extend the K\"ahler-Einstein metric $h_{KE}$ on $\mathcal{T}(S)$. In \cite{CY}, the Authors claim that bounded pseudoconvex domains in $\C^n$ with $\mathcal{C}^2$ boundary always admit a unique K\"ahler-Einstein metric. As we have pointed out above, $\mathcal{P}(S)$ is unbounded, hence we have no guarantee on the existence of a K\"ahler-Einstein metric without using other arguments. The case of those metrics that extend the McMullen metric on $\mathcal{T}(S)$ is different and we postpone the discussion in the next section.

\subsection{The Kobayashi and Carath\'eodory pseudodistances on $\mathcal{P}(S)$} In this paragraph we are going to consider the Kobayashi and Carath\'eodory pseudodistances on the moduli space of projective structure on $S$. Indeed, since $\mathcal{P}(S)$ is a complex manifold, both pseudodistances can be defined. Surprisingly, we have the following result.

\begin{thm}\label{kcnotcom}
Both Kobayashi and Carath\'eodory pseudodistances on $\mathcal{P}(S)$ can not be upgraded to a distance.
\end{thm}

\noindent The proof of this theorem is a straighforward consequence of the following proposition that we state in full generality.

\begin{prop}
Let $M$ and $N$ two complex manifolds. Then the Kobayashi and Carath\'eodory pseudodistances both satisfy the following chain of inequality.
\[ k_M(p_1,p_2)+k_N(q_1,q_2)\ge k_{M\times N}\big((p_1,q_1),(p_2,q_2)\big)\ge \max\{k_M(p_1,p_2),k_N(q_1,q_2)\},
\]
\[ c_M(p_1,p_2)+c_N(q_1,q_2)\ge c_{M\times N}\big((p_1,q_1),(p_2,q_2)\big)\ge \max\{c_M(p_1,p_2),c_N(q_1,q_2)\},
\] for any $p_1,p_2\in M$ and for any $q_1,q_2\in N$.
\end{prop}

\begin{proof}
Consider first the Kobayashi pseudodistance. We have that
\[ k_M(p_1,p_2)+k_N(q_1,q_2)\ge k_{M\times N}\big((p_1,q_1),(p_2,q_1)\big)+k_{M\times N}\big((p_2,q_1),(p_2,q_2)\big)\ge k_{M\times N}\big((p_1,q_1),(p_2,q_2)\big)
\] where the first inequality follows from the fact the mappings $f(p)=(p,q_1)$ and $g(q)=(p_2,q)$ are distance decreasing by the Pick-Schwarz lemma and the second inequality come from the triangle inequality. Finally the inequality
\[ k_{M\times N}\big((p_1,q_1),(p_2,q_2)\big)\ge \max\{k_M(p_1,p_2),k_N(q_1,q_2)\}
\] follows from the fact that the projections on both pieces are distance decreasing again by the Pick-Schwarz lemma. The same proof works replacing $k_M$ and $k_N$ with $c_M$ and $c_N$ respectively.
\end{proof}

\begin{proof}[Proof of Theorem \ref{kcnotcom}]
We argue by contradiction. By the proposition above, the Kobayashi pseudodistance on $\mathcal{T}(S)\times \C^{3g-3}$ is bounded from above by the sum of the Kobayashi pseudodistances $k_{\mathcal{T}(S)}$ and $k_{\C^n}$, and bounded from below by the maximum of them. Since $k_{\C^n}\equiv 0$, we get 
\[ k_{\mathcal{T}(S)\times \C^{3g-3}}\big((X,v),(Y,w)\big)=k_{\mathcal{T}(S)}(X,Y), 
\] for any $X,Y\in \mathcal{T}(S)$ and for any $v,w\in \C^n$. Since $k_{\mathcal{T}(S)\times \C^{3g-3}}\big((X,v),(X,w)\big)=0$ for any $v,w$ we  easily deduce that it is not a distance. Let $k_{\mathcal{P}(S)}$ be the Kobayashi pseudodistance on $\mathcal{P}(S)$, suppose it is a distance. Then the Schwarzian identification pulls-back this distance to a Kobayashi distance on $\mathcal{T}(S)\times \C^{3g-3}$, hence we get the desire contradiction. The same result follows for the Carath\'eodory pseudodistance applying \ref{compkc}. Equivalently, the same argument works for $c_{\mathcal{T}(S)\times \C^{3g-3}}$.
\end{proof}

\noindent By Theorem \ref{kcnotcom}, the moduli space $\mathcal{P}(S)$ is not a hyperbolic manifold in the sense of Kobayashi. 

\begin{rmk}
As a final remark we consider again the K\"ahler-hyperbolic metric $h_{MM}$ on $\mathcal{T}(S)$ and its extensions on $\mathcal{P}(S)$. Like in the case of $h_{KE}$, it would be interesting to know whether a K\"ahler-hyperbolic metric on $\mathcal{P}(S)$ exists or not. In \cite{GR}, Gromov showed that the notions of K\"ahler-hyperbolicity and Kobayashi-hyperbolicity are related in the compact case in the following way
\[ \text{K\"ahler-hyperbolicity } \Longrightarrow \text{ Kobayashi-hyperbolicity}
\] However, this implication does not hold in non-compact case. It would be interesting to understand if it holds in our situation. In such a case, by theorem \ref{kcnotcom} we can deduce that a K\"ahler-hyperbolic metric on $\mathcal{P}(S)$ does not exists.
\end{rmk}

\subsection{The Bergman pseudometric on $\mathcal{P}(S)$}\label{bpd}
Another question of major interest is whether the Bergman metric on $\mathcal{P}(S)$ exists or not and whether such a metric extends the Bergman metric on $\mathcal{T}(S)$. Since the moduli space of projective structures is biholomorphic to the product $\mathcal{T}(S)\times \C^{3g-3}$, it can be view as an unbounded domain inside $\C^{6g-6}$. A Bergman pseudometric on a domain $\Omega$ can be defined as soon as the Hilbert space of square integrable holomorphic functions is ample in some sense. Supposing that a Bergman pseudometric is defined, then it is classical in literature that for bounded domains the Bergman pseudometric is always a honest metric, but the same does not hold in the unbounded case. In this section we shall prove the following.

\begin{thm}\label{bnotcom}
The moduli space $\mathcal{P}(S)$ does not carry a Bergman metric.
\end{thm}

\noindent Before going to show this theorem, we recall some basic facts about the Bergman metric for a general domain $\Omega\subset \C^n$. Let 
$\Omega$ be a domain inside $\C^n$, and consider the Hilbert space $L^{2,h}(\Omega)$ of all square integrable holomorphic functions on $\Omega$. Any holomorphic square integrable function $f$ on $\Omega$ satisfies the estimate 
\begin{equation}\label{ineq}
 \sup_{K} |f(z)|\le C_K ||f||_{L^2(\Omega)}
\end{equation}
for any compact set $K$ inside $\Omega$. Inequality \ref{ineq} implies that for each $z\in\Omega$ the evaluation map
\[ \textsf{ev}_z: L^{2,h}(\Omega)\longrightarrow \C, \quad f\longmapsto f(z)
\] is a continuous linear functional on $L^{2,h}(\Omega)$. By Riesz representation theorem, this functional can be represented as the inner product with an element $\eta_z\in L^{2,h}(\Omega)$, which depends on $z$, so that
\[ \textsf{ev}_z(f)=\langle f, \eta_z\rangle=\int_\Omega f(\xi)\overline{\eta_z(\xi)}d\mu(\xi).
\] The Bergman kernel $k$ is defined as $k(z,\xi)=\overline{\eta_z(\xi)}$. Assume $k(z,z)>0$ for every $z\in\Omega$, that is, at every $z$ there exists a function $f\in L^{2,h}(\Omega)$ such that $f(z)\neq0$. In particular the quantity $\log k(z,z)$ is well-defined. The Bergman pseudometric $b_\Omega$ on $\Omega$ is defined as
\[ b_\Omega=2\sum_{i=1}^n a_{ij}dz^id\overline{z}^j \quad \text{ where } a_{ij}=\frac{\partial^2}{\partial z^i \partial \overline{z}^j}\log k(z,z).
\] The Bergman pseudometric is actually a honest metric as soon as the domain $\Omega$ is bounded (see for instance \cite{SK2}). Indeed: if $\Omega$ is bounded, the space $L^{2,h}(\Omega)$ contains all polynomial functions and then $\Omega$ admits a honest Bergman metric. A couple of remarks.

\SetLabelAlign{center}{\null\hfill\textbf{#1}\hfill\null}
\begin{enumerate}[leftmargin=1.75em, labelwidth=1.3em, align=center, itemsep=\parskip]
\item[\bf 1.] Recall that the Teichmüller space $\mathcal{T}(S)$, of a closed genus $g$ surface $S$, can be realize as a domain of holomorphy in $\C^{3g-3}$ by Bers' embedding. Nehari-Kraus' theorem implies that such domain is bounded; hence it carries a Bergman metric $b_\Omega$ which turns out K\"ahler complete on $\mathcal{T}(S)$.
\item[\bf 2.] In the case of $\C^n$, the Hilbert space $L^{2,h}(\C^n)$ is trivial. Indeed, the only square integrable holomorphic function on $\C^n$ is the zero function. As a consequence, the complex space $\C^n$ does not carry any Bergman metric nor Bergman pseudometric since the Bergman kernel $k(z,z)$ is trivially the zero function.
\end{enumerate}

\noindent Let $\Omega_1$ and $\Omega_2$ be complex domains inside $\C^n$ and $\C^m$ respectively. Their product $\Omega_1\times\Omega_2$ is a domain inside the complex space $\C^{n+m}$. The following product formulas for the Bergman metric is weIl known and the proof can be found, for instance, in \cite[Proposition 4.10.17]{SK2}.

\begin{prop}\label{pbm}
Let $\Omega_1$ and $\Omega_2$ be complex domains inside $\C^n$ and $\C^m$ respectively, then
\[ b_{\Omega_1\times\Omega_2}=b_{\Omega_1}\times b_{\Omega_2}.
\]
\end{prop}

\noindent The proposition says that if the Bergman metric on $\Omega_1\times \Omega_2$ exists, then it can be split as the direct product of the Bergman metrics on $\Omega_1$ and $\Omega_2$ respectively. In particular, if one (possibly both) of these metrics is not defined, then the Bergman metric on $\Omega_1\times \Omega_2$ does not exists.\\

\noindent Coming back to our moduli space $\mathcal{P}(S)$, if a Bergman metric (or pseudometric) exists on such space, by \ref{pbm} it induces a Bergman metric (pseudometric) on $\C^n$ which is actually not defined by the previous remark. Hence the moduli space $\mathcal{P}(S)$ does not carry a Bergman metric nor pseudometric.

\printbibliography
\end{document}